\documentclass[11pt,leqno]{article}
\usepackage{graphicx, amsfonts, amsthm, amsxtra, amssymb, verbatim, makeidx}
\usepackage{subeqnarray, relsize}
\usepackage[mathscr]{euscript}
\usepackage{hyperref}
\makeindex
\makeglossary
\begin{document}
\newtheorem{theo}{Theorem}
\newtheorem{exam}{Example}
\newtheorem{coro}{Corollary}
\newtheorem{defi}{Definition}
\newtheorem{prob}{Problem}
\newtheorem{lemm}{Lemma}
\newtheorem{prop}{Proposition}
\newtheorem{rem}{Remark}
\newtheorem{conj}{Conjecture}
\newtheorem{calc}{}
\newtheorem{proper}{Property}

%----------------General Math Notations-------------------------------------
\def\Z{\mathbb{Z}}                   %Integer  numbers
\def\Q{\mathbb{Q}}                   %Rational  numbers
\def\C{\mathbb{C}}                   %Complex numbers
\def\N{\mathbb{N}}                   %natural numbers
\def\P{\mathbb{P}}                   %Projective space
\def\uhp{{\mathbb H}}                %upper half plane
\def\A{\mathbb{A}}                   %affine space C^n
\def\dR{{\rm dR}}                    %The subindex dR standing for de Rham cohomology.
\def\F{{\cal F}}                     %A foliation
\def\Sp{{\rm Sp}}                    %Symplectic group
\def\Gm{\mathbb{G}_m}                 %The multiplicative group
\def\Ga{\mathbb{G}_a}                 %The additive  group
\def\Tr{{\rm Tr}}                      %Trace map in Algebraic de Rham cohomology
\def\tr{{{\mathsf t}{\mathsf r}}}                 %Transposition of matrices
\def\spec{{\rm Spec}}            %The spectrume
\def\ker{{\rm ker}}              %kernel
\def\GL{{\rm GL}}                %The liner group
\def\ker{{\rm ker}}              %kernel
\def\coker{{\rm coker}}          %cokernel
\def\im{{\rm Im}}               %Image
\def\coim{{\rm Coim}}            %coimage
\def\T{{\sf T }}                    %a parameter space
\def\S{{\sf S }}                    %a parameter space
\def\V{{\sf V }}                    %a parameter space
\def\W{{\sf I  }}                %Inciden
\def\X{{\sf X }}                 %The total space of a family of projective varieties...
\def\Y{{\sf Y }}                 %The total space of a family of algebraic cycles....
\def\Zc{{\sf Z }}                 %The total space of a family of algebraic cycles....
\def\Resi{{\rm Resi }}            %Residue
\def\U{{\cal U }}                 %covering
\def\O{{\cal O }}                 %covering
%----------------------------------------------------------------

\begin{center}
{\LARGE\bf 
%Periods of families of curves in threefolds
Periods of families of curves in threefolds
}
\\
\vspace{.25in} {\large {\sc Hossein Movasati}}
\footnote{
Instituto de Matem\'atica Pura e Aplicada, IMPA, Estrada Dona Castorina, 110, 22460-320, Rio de Janeiro, RJ, Brazil,
{\tt www.impa.br/$\sim$ hossein, hossein@impa.br}}
\end{center}

\begin{abstract}
Clemens' conjecture states that the the number of rational curve in a generic quintic threefold is finite. If it is false we prove that certain periods of rational curves in such a quintic threefold must vanish.
%therefore, we reduce the problem from a generic quintic to a specific quintic like Fermat quintic which has families of rational curves. 
Our method is based on a generalization  of a proof of Max Noether's 
theorem using infinitesimal variation of Hodge structures and its reformulation in terms of integrals and Gauss-Manin connection. 
\end{abstract}
\section{Introduction}
Clemens in \cite[page 300]{Clemens1984} in his study of  Griffiths' Abel-Jacobi mapping,  conjectures that the number of rational curves of degree $d$ in a generic quintic threefold must be finite. The main motivation for this is a counting of equations and parameters of a rational curve inside a quintic and its verification for degree $1$ and $2$ curves by Sh. Katz. Since then the conjecture is proved for $d\leq 11$. For an overview and many related references see \cite{Cotterill2012, ClemensConj}.
%\footnote{Clemens' conjecture. Encyclopedia of Mathematics. URL: http://encyclopediaofmath.org/index.php?title=Clemens%27_conjecture&oldid=42180 }
In this note we discuss a possible generalization of a proof of Noether-Lefschetz theorem in the framework of families of curves in quintic threefolds. It states that for a  generic smooth surface $X$ in $\P ^3$, the only curves in $X$ are obtained by intersecting $X$ with another surface. Equivalently, the Picard number of  $X$ is equal to one. The infinitesimal variation of Hodge structures (IVHS) developed by Ph. Griffiths and his coauthors in \cite{CGGH1983} gives a rigorous proof of Noether-Lefschetz theorem. In \cite[Section 14.4]{ho13} we have rewritten this proof using integrals (periods). The main idea lies in the fact that the restriction of holomorphic two forms over a curve is identically zero, and hence its integration over the curve is zero. If we have a full family of such surfaces and curves then we can make derivation of such integrals with respect to the underlying parameters (this is the origin of Gauss-Manin connection and IVHS) and we conclude that  the integration of all elements of the primitive cohomology of $X$ over the curve is zero. This concludes the proof. In the present text we mimic this proof and this leads us to our main result.  
\begin{theo}
 \label{main2021}
 If Clemens' conjecture is false in degree $d$ then for a holomorphic family of degree $d$ rational curves $Z_s,\ s\in(\C,0)$ in a generic  quintic threefold  $X$ and any 
 $\omega\in F^2H^3_\dR(X)$ we have 
 \begin{equation}
 \label{3oct2021}
 \int_{Z_s}\frac{\omega}{ds}=0, 
 \end{equation}
 where $F^\bullet$ is the Hodge filtration of $H^3_\dR(X)$.
 \end{theo}
 By an analytic  family of algebraic cycles $Z_s,s\in(\C,0)$ in $X$ we mean the following. We take an analytic variety $Z$,  holomorphic maps $f:Z\to (\C,0)$ and $g: Z\to X$  such that $f$ is proper and for all $s\in (\C,0)$, $g$ restricted to $Z_s:=f^{-1}(s)$ is an injection. 
 The Gelfand-Leray form $\frac{\omega}{ds}$ in our context is only defined for $\omega \in F^2H^3_\dR(X)=H^{30}\oplus H^{21}$, for which the pull-back of $\omega$ to the two dimensional variety $Z$ by $f$ can be divided by $df$.  
Similar to the case of Abel-Jacobi map, we also use the integration over the path of homology between two algebraic cycles. According to \cite{Clemens1984} this goes back at least to \cite[page 333]{Weil1946} which in turn must be inspired by Picard's intensive study of two dimensional integrals in \cite{Picard1900}. The proof of Theorem \ref{main2021} actually implies that the Abel-Jacobi map attached to $Z_s,s\in(\C,0)$ is identically zero, however, this is much weaker statement than the vanishing \ref{3oct2021}, see Section \ref{10jan2022}. The integration over algebraic cycles as in \eqref{3oct2021} appears first in Deligne's study of absolute Hodge cycles, see \cite{dmos}, and further applications of this in the study of Hodge loci has been initiated in \cite[Chapters 18,19]{ho13} \cite{Roberto}.  

The sketch of the proof of Theorem \ref{main2021} is as follows. If a generic quintic contains an infinite number of genus $g$ curves then  we have to rigorously define families of rational curves  which vary inside families of quintic threefolds. This is done in Section \ref{Hilbert2021} for rational curves. This is the only place which we consider $g=0$ and it is expected that Theorem \ref{main2021} is true for curves of any genus. For instance, in Section \ref{10jan2022} we show that for curves which are obtained by the intersection of $X$ with $\P^2\subset\P^4$ we have automatically \eqref{3oct2021}.  We then study integrals over homology paths between two curves and give a formula for its derivation, see Section \ref{gm2021}. Some of these integrals are identically zero. This is described in Section \ref{04.12.2021-1}. Derivating these integrals we get more vanishing integrals and this finishes the proof in Section \ref{3.10.2021}. In Section \ref{04.12.2021-2} we describe how one can compute integrals \eqref{3oct2021} and as an example we perform this computation for a well-known family of lines inside  Fermat quintic.  The condition $X$ to be generic in Theorem \ref{main2021} comes from our difficulties in Section \ref{Hilbert2021} to handle families of algebraic cycles. We expect that for any family of rational curves inside any smooth quintic threefold, \eqref{3oct2021} is not identically zero in $s$.  

I presented the main ideas of the present paper in an informal online seminar and I would like to thank the audience for their comments and questions. This includes Ethan Cotterill who introduced me to the recent preprint \cite{Mustata2021},  Roberto Villaflor with whom I had useful discussions regarding the Abel-Jacobi map, Felipe Ramos who pointed out the correct definition of Hodge filtration in relative de Rham cohomology in Section \ref{04.12.2021-1},   and  Younes Nikdelan, Jin Cao, Jorge Duque. Finally, many thanks go to P. Deligne who wrote two useful letters for the first draft of the present paper. These can be found in the author's webpage.

\section{Hilbert schemes}
\label{Hilbert2021}
In the following we will discuss Hilbert schemes. We will not use scheme structure of these objects; only the underlying analytic variety will be used. 
Let $\T$  be an open subset of a Hilbert scheme parametrizing smooth projective varieties $X\subset \P^N$ of dimension $2n+1$. 
Let also $\S$ be another component of a Hilbert scheme of projective varieties 
$Z\subset\P^{N}$ of dimension $n$. We consider the incidence scheme 
$$
\W:=\{(X,Z)\in \T\times\S | Z\subset X\}.
$$
 We fix an irreducible component $\tilde \W$ of $\W$, consider the projections $\pi: \W\to\T,\ \tilde\W\to\T$ and denote their  images by $\V$ and $\tilde\V$, respectively.  From now on we denote by $t$ a point in $\T$ and by $X_t$ the corresponding projective space. Our main example is $\T:=\P\left( \C[x_0,x_1,\ldots,x_{n+1}]_a\right)$  with $n=3,\ a=5$ and the following with $N=4$ which is not actually a Hilbert scheme but it will be enough for our purposes.  A rational curve is given by the image of a map  
$$
f:\P^1\to \P^N, [x:y]\mapsto [f_0(x,y):\cdots:f_N(x,y)],
$$
where $f_i(x,y)$ are homogeneous polynomials of degree $d$ in $x,y$. We may expect that 
the quotient 
\begin{equation}
\P\left(\C[x,y]_d^{N+1}-\C[x,y]_d\C^{N+1}\right)/{\rm PSL}(2,\C),
\end{equation}
exists as an algebraic variety over $\C$, where $\C[x,y]\C^{N+1}$ corresponds to those $f$ such that its image is a point.  We may also expect that the universal family over $\S$ exists. This is  
$$
\left(\P^1\times \P\left(\C[x,y]_d^{N+1}-\C[x,y]_d\C^{N+1}\right)\right)/{\rm PSL}(2,\C),
$$
where the action of $A\in {\rm PSL}(2,\C)$ in $\P ^1$ is given by the  inverse of $A$, and hence, we have well-defined map 
$$
F: {\sf P}\to \P^N, \ \ \ \ [x:y],(f_0,f_1,\ldots,f_N)\mapsto [f_0(x,y);f_1(x,y):\cdots: f_N(x,y)].
$$
We have a natural projection ${\sf P}\to \S$ such that its fibers are rational curves and $F$ restricted to these fibers restores the map $f$. 
E. Cotterill reminded me that the moduli space $\S$ has a natural stratification given by the geometric genus  of the rational curve. Each strata might be isomorphic to a   Hilbert scheme  of rational curves of degree $d$ 
 in $\P^{N}$. It might be clarifying to write down the details of all this. 
We avoid these details and simply  define: 
\begin{equation}
\label{11oct2021}
\S:=\P\left(\C[x,y]_d^{N+1}-\C[x,y]_d\C^{N+1}\right),\ 
{\sf P}:=\P^1\times \S.
\end{equation}
In this way we have an action of ${\rm PSL}(2,\C)$ on the incidence scheme $\W$ such that it leaves the fibers of $\W\to\T$ invariant. In this example it is known that $\W$ is irreducible for $d\leq 11$, however, for $d\geq 12$ it might have many irreducible components, see \cite{Cotterill2012} and the references therein.
\begin{figure}
\begin{center}
\includegraphics[width=0.5\textwidth]{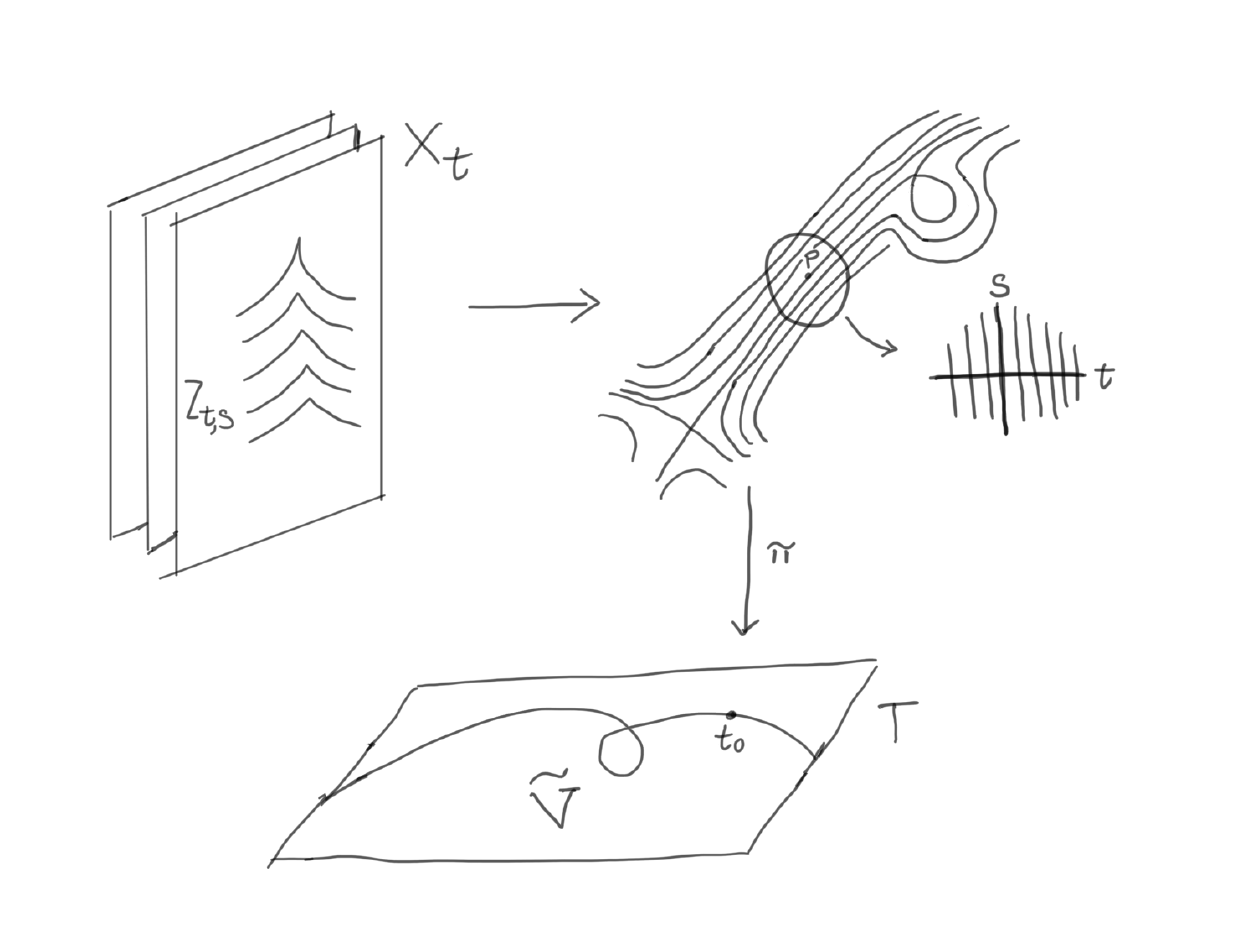}
\caption{Incidence scheme}
\label{dfd}
\end{center}
\end{figure}
We take  smooth points $t_0\in\tilde\V,\ p\in\tilde\W$  with $\pi(p)=t_0$ and $p$ is a regular point of $\pi$, that is, the derivative of $\pi$ at $p$ is surjective. This follows from:
 Let $\pi: (\C^n,0)\to (\C^m,0)$ be a surjective holomorphic map. Then the locus of points $p\in (\C^n,0)$ such that the derivative of $\pi$ at $p$ is not surjective, is a proper analytic subset of $(\C^n,0)$, and hence its complement is dense.   By implicit function theorem, we can take coordinate system $(t,s)\in (\C^n\times \C^m,0)\cong (\tilde\W,p)$ and $t\in (\C^n,0)\cong\V$ such that $\pi$ is just projection in the $t$ coordinate, see Figure \ref{dfd}.  We get a family $(X_{t}, Z_{t,s})$ such that $Z_{t,s}\subset X_t$. 
 \begin{rem}\rm
 \label{09.12.2021}
In the context of Hilbert schemes this family is flat. For simplicity, we start with a pair $(X,Z)$ such that both $X$ and $Z$ are irreducible, and hence, we can assume that $X_{t}$ and $Z_{t,s}$ are irreducible. 
In our main example \ref{11oct2021}, since we have also the action of ${\rm PSL}(2,\C)$ on the fibers of $\tilde\W\to \T$ and this action has only discrete (dimension zero) stabilizers, we replace $(\C^m,0)$ with some linear subspace $(\C^{m-3},0)$ which intersects the orbits of ${\rm PSL}(2,\C)$ in discrete sets. 
\end{rem}
If $\V$ is a proper subset of $\T$ ($\pi$ is not dominant) then a generic $X$ does not contain any algebraic cycle $Z$. 
  If $\T=\tilde\V$ and $m=0$ for all components $\tilde\W$ of $\W$ then we have a finite number of $Z$ inside a generic $X$.
  Note that $n=\dim(\tilde\V)$ and $n+m$ is the dimension of $\tilde\W$.
From now on we consider the case $m>0$. For our main example this means that we assume that Clemens' conjecture does not hold. Our main goal in this section is to describe instances such that the following property holds:
\begin{proper}
\label{F.France2021}
We have 
\begin{enumerate}
 \item 
 The homology classes of $Z_{t, s_1},Z_{t, s_2}$ in $H_*(X_t,\Z)$ for $s_1,s_2\in(\C^m,0)$ and $t\in (\C^n,0)$ are the same. The homology path $\delta_t$ between $Z_{t, s_1}$ and $Z_{t, s_2}$ is supported in the family of algebraic cycles $\cup_{s\in (\C^m,0)} Z_{t,s}\subset X_t$. 
 \item
 For $s$ fixed and $t_1,t_2\in (\C^n,0)$, the homology class of $Z_{t_1,s}$ is mapped to the homology class of $Z_{t_2,s}$ under the monodromy map $H_*(X_{t_1},\Z)\to H_*(X_{t_2},\Z)$. 
\end{enumerate}
\end{proper}
Even though the varieties $X_t$ are $C^\infty$ isomorphic, the algebraic cycles $Z_{t_1,s}, Z_{t_2,s}$ might not be isomorphic topologically. For instance, $Z_{t_1,s}$ might be a curve with only one nodal singularity and $Z_{t_2,s}$ a curve with only one cuspidal singularity. 
\begin{prop}
\label{30sept2021toulouse}
If $\S$ parameterizes rational curves of degree $d$ as in \eqref{11oct2021} then 
Property \ref{F.France2021} holds.   
\end{prop}
\begin{proof}
We have an small open smooth set $U$ of $\tilde \W$ with the coordinate system $(t,s)$. 
%Consider the projection $a: U\to\S$. If necessary we replace $U$ with a smaller open set, and assume that  we have a holomorphic section $b$ of the projection $\left(\C[x,y]_d^{N+1}-\C[x,y]_d\C^{N+1}\right)\to \S$ defined in a neighborhood of $a(U)$.   
We get a holomorphic family  $f_{t,s}: \P^1\to \P^N, [x:y]\mapsto [f_0(x,y):\cdots:f_N(x,y)]$, where $f_i(x,y)$ are homogeneous polynomials of degree $d$ in $x,y$ depending holomorphically in $(t,s)$. Such $f_i$'s are just the ingredients of the projection of $(t,s)$ under $U\to\S$ (recall the definition of $\S$ in \ref{11oct2021}).
%Such $f_i$'s are obtained by 
%$$
%b\circ a(t,s)=(f_0,f_1,\ldots,f_N).
%$$
By definition the image of $f_{t,s}$ is $Z_{t,s}$.  We get map $f_{t,s,*}: H_2(\P ^1,\Z)\to H_2(X_t,\Z)$, and since $\Z$ is discrete, we conclude the desired property. By our considerations in Remark \ref{09.12.2021}, for fixed $t$ the family of algebraic cycles $Z_{t,s}$ is not constant in $s$.  
\end{proof}
From now on we replace $(\C^m,0)$ with a one dimensional subspace and assume that $m=1$. 
\begin{defi}\rm 
\label{3nov2021}
Let $X$ be a smooth projective variety and $Z_1,Z_2\subset X$ be two algebraic cycles. We say that $Z_1$ is strongly algebraically equivalent to $Z_2$ if we have a holomorphic family of algebraic cycles $Z_s, s\in (\C,0)$ such that 
\begin{enumerate} 
 \item 
$Z_1$ and $Z_2$ are two members of this family.
\item 
The homology path $\delta$ between $Z_{1}$ and $Z_{2}$ is supported  in the family of algebraic cycles $\cup_{s\in (\C,0)} Z_{s}\subset X$.
\end{enumerate}
\end{defi}
\begin{rem}\rm
From the first item of the above definition it follows that the homology classes of $Z_{1},Z_{2}$ in $H_*(X,\Z)$  are the same, see for instance  \cite[Proposition 19.1.1, page 373]{Fulton1998} for this implication. It is natural to expect that the second item holds automatically, however, this does not seem to be the case. The proof in the mentioned reference is not adaptable to provide this stronger statement. 
\end{rem}
Recall Property \ref{F.France2021}. For fixed $t$ and $s_1,s_2\in (\C^m,0)$, $Z_{1,t}:=Z_{t,s_1}$ and $Z_{2,t}:=Z_{t,s_2}$ are strongly algebraically equivalent. Definition \ref{3nov2021} combines topological and algebraic ingredients, whereas the the classical definition of algebraic equivalence in algebraic geometry is purely algebraic.  
%see for instance \cite{Kleiman1994}, are purely algebraic. 
It would be instructive to compare these definitions.

%\begin{rem}\rm
%\label{4nov2021}
% The set $\cup_{s\in (\C,0)} Z_{s}$ is not an analytic subvariety of $X$ as it is the image of a holomorphic map $g: Z\to X$ (see the definition after Theorem \ref{main2021}). There is an open dense subset of this image which is an analytic subvariety of $X$ of dimension $n+1$, where $n$ is the dimension of $Z_s$. If we replace $(\C,0)$ with a compact algebraic curve, then such an image is always a subvariety of dimension $n+1$ of $X$.   
%\end{rem}

\begin{figure}
\begin{center}
\includegraphics[width=0.5\textwidth]{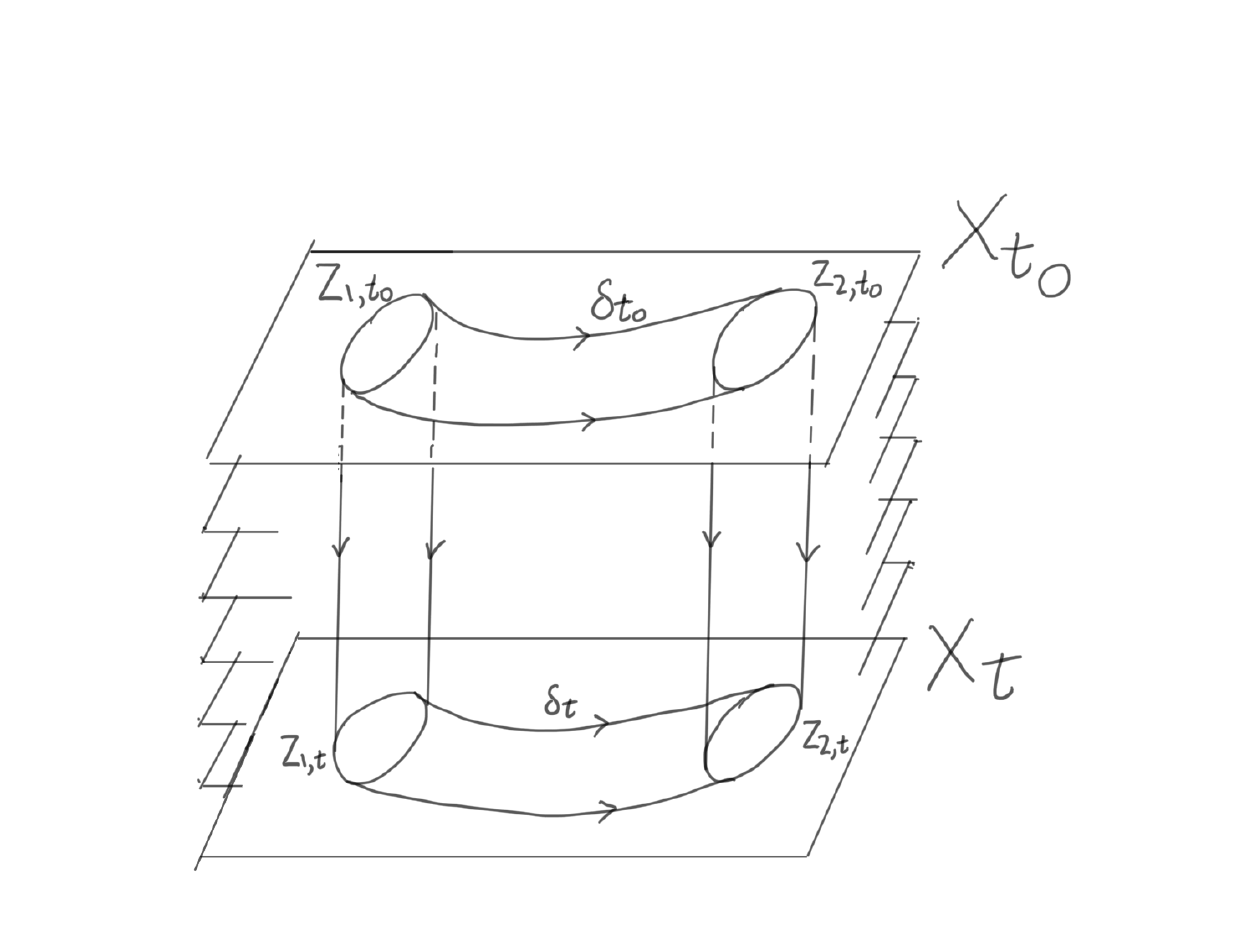}
\caption{Gauss-Manin connection }
\label{gm}
\end{center}
\end{figure}

\begin{rem}\rm
A more general definition is as follows: 
Let $X$ be a smooth projective variety and and  $Z_1,Z_2\subset X$ be two irreducible subvariety of $X$. We say that  $Z_1$ and $Z_2$ are flat algebraically equivalent if we have have holomorphic maps  $f: Z\to (\C^m,0)$ and $g: Z\to X$ such that $f$ is flat and $g$ restricted to the fibers of $f$ is an injection.
Duco van Straten kindly reminded me [January 14, 2019] 
that
``the requirement of flatness of a
cycle is not a reasonable condition at
all. Flatness is just too special. 
There is the notion of families of
cycles used by Barlet in analytic
and by Kollar in algebraic geometry[see \cite{kollar1995}, page 45-46 and \cite{Barlet1975}].
It is very ugly and algebraically hard
to use, but geometrically reasonable." 
For our purpose we need Proposition \ref{30sept2021toulouse} which is very specific to rational curves. One might expect that if two irreducible algebraic cycle $Z_1,Z_2\subset X$  are flat algebraically equivalent then they are homologically equivalent. 
It is not known to the author whether this is true for rational curves.
\end{rem}
\begin{rem}\rm
\label{6dec2021}
 We take two transversal sections $\Sigma_1$ and $\Sigma_2$ in $(\W,p)$ to the fiber of $\pi$ such that under projection $\pi$ they are biholomorphic to $(\C^n,0)$.
For instance,  $\Sigma_i:=\{s=s_i\}$ for two values $s_1$ and $s_2$ of $s$. 
We have a family $Z_{1,t}, Z_{2,t}\subset X_t,\ t\in (\C^n,0)$ such that $Z_{1,t}:=Z_{t,s_1}$ and $Z_{2,t}:=Z_{t,s_2}$ are strongly algebraically equivalent,  see Figure \ref{gm}.
\end{rem}
%\begin{rem}
%If we fix a point $t_0\in\T$, for instance take $X_{t_0}$ the Fermat variety, it might happen the fiber of $\tilde\W\to \T$ over $t_0$ lies $\tilde\W$. 
%\end{rem}

\section{Gauss-Manin connection}
\label{gm2021}
In this section we write down a formula which can be formulated in terms of the Gauss-Manin connection in relative cohomology, however, for the sake of simplicity and avoiding unnecessary notations, we present it as derivating integrals over non closed homology classes. 

Let us consider a projective variety $X$ of dimension $2n+1$ and two algebraic cycles $Z_1,Z_2\subset X$ of dimension $n$ such that their homology class $[Z_1]$ and $[Z_2]$ are equal in $H_{2n}(X,\Z)$, that is, $Z_1$ is homologous to $Z_2$. Let $Y=Z_1\cup Z_2$ and $\delta\in H_{2n+1}(X,Y)$ be the homology between $Z_1$ and $Z_2$, that is, it is mapped to 
$[Z_2]-[Z_1]\in H_{2n}(Y,\Z)$ under the boundary map $H_{2n+1}(X,Y)\to H_{2n}(Y)$.

We consider a one parameter holomorphic family $(X_t, Y_t),\ t\in (\C,t_0)$ of such objects, that is we have a proper submersion $\pi:\X\to (\C,t_0)$ and 
$\Y=\Zc_1\cup\Zc_2\subset \X$ with the properties of fibers as above.   
From this we mainly mean the family 
$(X_{t}, Z_{1,t}\cup Z_{2,t})$ constructed in Proposition \ref{30sept2021toulouse}, see also Remark \ref{6dec2021}.
The subindex $t$, for instance $\delta_t\in H_{2n+1}(X_t,Y_t)$, will be used in order to emphasize the dependence in $t$. For a complex manifold $M$, $\Omega^{i}_{M^\infty}$ and $\Omega^{i}_{M}$ denotes the set of global $C^\infty$ and holomorphic $i$-forms in $M$, respectively.
Let $\omega\in\Omega^{2n+1}_{\X^\infty}$ and $\alpha\in\Omega^{2n}_{\Y^\infty}$ such that 
$$
d\omega\in \pi^{-1}\Omega^1_{(\C,t_0)}\wedge\Omega^{2n+1}_{\X^\infty}, \
\omega|_{\Y}-d\alpha\in \pi^{-1}\Omega^1_{(\C,t_0)}\wedge\Omega^{2n}_{\Y^\infty}. 
$$
In the following theorem we have assumed that the integrals depends holomorphically in $t$. One can achieve this by considering algebraic de Rham cohomology and it is  the case when $X_t$'s are hypersurfaces. 
\begin{theo}
\label{7set2021bolsonaro}
We have  
\begin{equation}
\label{khodayakomak} 
\frac{\partial}{\partial t}\int_{\delta_t}\omega=\int_{\delta_t}\frac{d\omega}{dt}+
\int_{Z_{2,t}}\frac{\omega|_{\Y}-d\alpha}{dt}-\int_{Z_{1,t}}\frac{\omega|_{\Y}-d\alpha}{dt}+
\frac{\partial}{\partial t}\left(
\int_{Z_{2,t}-Z_{1,t}}\alpha\right).
\end{equation}
\end{theo}
\begin{proof}
We define $D_t$ to be the locus in $\X$ which $\delta_t$ sweeps: 
$$
D_t=\cup_{s\in [0,1]} \delta_{\gamma_t(s)}\subset \X
$$ 
such that $\gamma$ is a path in $(\C,t_0)$ 
connecting $t_0$ to $t$, see Figure \ref{gm}. 
%and $\delta_{\gamma_t(s)}$ is the trace of $\delta_t$ when it vanishes along $\gamma_t$. 
Note that $D_t$ is oriented in a canonical way such that 
$$
\partial D_t:=(\pi|_{\Zc_1})^{-1}(\gamma)+\delta_t-(\pi|_{\Zc_2})^{-1}(\gamma) -\delta_{t_0}=
\gamma_1+\delta_t-\gamma_2-\delta_{t_0},
$$
see Figure \ref{gm}. 
By Stokes theorem we have 
\begin{eqnarray*}
\int_{\delta_t}\omega &=&\int_{D_t}d\omega+\int_{\delta_{t_0}}\omega+ \int_{\gamma_2}(\omega|_{\Zc_2}-d\alpha)-\int_{\gamma_1}(\omega|_{\Zc_1}-d\alpha)+
\int_{\gamma_2-\gamma_1}d\alpha.\\ 
%&=&
%\int_{D_t}\alpha_1\wedge \omega_1+\\ 
\end{eqnarray*}
Taking the differential $\frac{\partial}{\partial t}$ of the above equality and using  the  Fubini's theorem we get the desired result.
\end{proof}

\section{A kind of Hodge cycle}
\label{04.12.2021-1}
For a smooth projective variety $X$, the integration of $F^{\frac{n}{2}+1}H^n_\dR(X)$ over an algebraic cycle of dimension $\frac{n}{2}$ in $X$ is zero, and this simple fact leads us to the notion of Hodge cycle and then the Hodge conjecture, see \cite[Section 5.13]{ho13-Roberto}. In this section we follow a similar argument to distinguish between algebraic equivalence of algebraic cycles in the sense of Definition \ref{3nov2021} and homology between algebraic cycles. For a pair of algebraic varieties $Y\subset X$, the relative algebraic de Rham cohomology is defined as the hypercohomology of complexes $H^m_\dR(X,Y):=\uhp^m(X, \Omega^\bullet_{X,Y})$,
where $\Omega^m_{X,Y}:=\Omega^m_X\times \Omega^{m-1}_Y$ is equipped with 
$d: \Omega^m_{X,Y}\to \Omega^{m+1}_{X,Y},\ d(\omega,\alpha):=(d\omega, \omega|_Y-d\alpha)$. 
A naive definition of the Hodge filtration $F^iH^m_\dR(X,Y):=\uhp^m(X, \Omega^{\bullet\geq i}_{X,Y})$ is not correct and it turns out that the correct definition compatible with long exact sequences is:
$F^iH^m_\dR(X,Y):={\rm Image}(\uhp^m(X, F^i\Omega^{\bullet}_{X,Y})\to \uhp^m(X, \Omega^{\bullet}_{X,Y}))$, where 
\begin{equation}
 \label{10022022}
(F^i\Omega^{\bullet}_{X,Y})^m:=\left\{
\begin{array}{cc}
 0 & m<i \\
 \Omega^m_{X} & m=i\\
 \Omega^m_{X,Y} & m>i
\end{array}\right.
,
\end{equation}
see \cite[Theorem 3.22]{SP2008}.
We will consider a good covering $\{U_i\}_{i\in I}$ of $X$ and consider elements of cohomologies relative to this covering. For more details see  \cite{ho13-Roberto}.     
\begin{prop}
\label{10.09.2021Ethan}
 Let $X$ be a smooth projective variety of dimension $2n+1$ and let $Z_1,Z_2$ be two subvarieties of $X$ of dimension $n$. If $Z_1$ and $Z_2$ are strongly algebraically equivalent (in the sense of Definition \ref{3nov2021}) then 
 \begin{equation}
 \label{10sept2021}
 \int_{\delta}F^{n+2}H^{2n+1}_\dR(X,Y)=0,
 \end{equation}
 where $\delta$ is the homology between $Z_1$ and $Z_2$. %(see  Proposition \ref{30sept2021toulouse}). 
\end{prop}
\begin{proof} 
The homology path $\delta$ between $Z_{1}$ and $Z_{2}$ is supported in the family of algebraic cycles 
$\cup_{s\in (\C,0)} Z_{s}$ which is the image of a holomorphic map $g: Z\to X$ with $\dim(Z)=n+1$.  
%and an open dense subet of  this lies inside an $n+1$ dimensional analytic subvariety $Z$ of $X$, see Remark \ref{4nov2021}. 
By definition, $F^{n+2}H^{2n+1}_\dR(X,Y)$ contains only differential $i$-forms with $i>n+1$. Therefore, the pull-back of this piece of the Hodge filtration to $Z$ is identically zero, and hence, its restriction to ${\rm Image}(g)$ and its  integration over $\delta$ is zero.   
\end{proof}
\begin{rem}\rm 
The vanishing \eqref{10sept2021} follows the same principle which has produced the Hodge conjecture, see \cite[Proposition 5.12, Section 5.13]{ho13-Roberto}. One might speculate (similar to the Hodge conjecture) that its inverse is true, that is, if $Z_1$ and $Z_2$ as above are homologous and we have \eqref{10sept2021} then they are  strongly algebraically equivalent.
Note that there might not exist closed cycles $\delta\in H_{2n+1}(X,\Q)$ such that \eqref{10sept2021} happens, in other words we might have $H^{2n+1}(X,\Q)\cap \left (H^{n,n+1}\oplus H^{n+1,n}\right)=0$.  Therefore, our  speculation does not follow from reformulation of the generalized Hodge conjecture by A.  Grothendieck. 
\end{rem}
\begin{prop}
\label{10.09.2021France}
 Under the canonical map $H^{2n+1}_\dR(X,Y)\to H^{2n+1}_\dR(X)$, the piece of Hodge filtration  $F^iH^{2n+1}_\dR(X,Y)$ is mapped isomorphically to $F^iH^{2n+1}_\dR(X)$ for $i\geq n+1$. For 
 $i\leq n$, the kernel of $F^iH^{2n+1}_\dR(X,Y)\to F^i_\dR H^{2n+1}(X)$ is one dimensional. 
\end{prop}
\begin{proof}

We have the long exact sequence 
\begin{equation}
 H^{2n}_\dR(X)\to H^{2n}_\dR(Y)\to H^{2n+1}_\dR(X,Y)\to H^{2n+1}_\dR(X)\to H^{2n+1}(Y)=0
\end{equation}
and hence $ H^{2n+1}_\dR(X,Y)\to H^{2n+1}_\dR(X)$ is surjective. Moreover, $H^{2n}_\dR(Y)$ is of dimension $2$ generated by the cohomology classes $[Z_1]$ and $[Z_2]$. After writing the above long exact sequence in homology and using the fact that $Z_1$ and $Z_2$ are homologous in $X$ through $\delta$, we observe that the integration of the images $x_1$ and $-x_2$ of $[Z_1]$ and $-[Z_2]$ in $H^{2n+1}_\dR(X,Y)$ over $\delta$ are the same and over the image of $H_{2n+1}(X)\to H_{2n+1}(X,Y)$ are zero. This implies that  the image of $H^{2n}_\dR(Y)\to H^{2n+1}_\dR(X,Y)$ is one dimensional generated by $x_1=-x_2$.
We take a good covering $\U$ of $(X,Y)$.
% by $2n+2$ Zariski open sets.
It turns out that an element of $H^{2n+1}_\dR(X,Y)$ relative to this covering is of the form
\begin{equation}
\label{22sept2021}
\omega=(\omega^{0},0)+(\omega^1,\alpha^0)+\cdots+
(\omega^n,\alpha^{n-1})+(\omega^{n+1},\alpha^n)+
(\omega^{n+2},0)+\cdots+(\omega^{2n+1},0)
\end{equation}
$$
(\omega^j,\alpha^{j-1})\in C^{2n+1-j}(\U,\Omega^j_{X,Y}),
$$
We have $\alpha^{j}=0,\ \ j=n+1,\cdots, 2n$ because $\dim(Y)=n$. This implies the statement on the pieces of Hodge filtration for 
$i\geq n+2$. For $i=n+1$ we note that $\alpha^n=0$ by our definition of Hodge filtration in \eqref{10022022}.
\end{proof}
% For $i=n+1$ we can say a little bit more. 
% \begin{prop}
% \label{portadosfundos2021}
% We have 
% $$
% F^{n+1}H^{2n+1}_\dR(X,Y)=\check F^{n+1} H^{2n+1}_\dR(X)\oplus \check H^{2n}_\dR(Y)
% $$
% where
% \begin{eqnarray*}
% \check F^{n+1} H^{2n+1}_\dR(X) &:=&\left\{  \omega\in F^{n+1}H^{2n+1}_\dR(X,Y)  | \ \alpha^{j}=0,\ \ j=0,1,2,\cdots,  n  \right\},\\
% \check H^{2n}_\dR(Y)&:=&{\rm Image}\left( H^{2n}_\dR(Y)\to H^{2n+1}_\dR(X,Y)\right).
% \end{eqnarray*}
% \end{prop}
% \begin{proof}
% In the Cech cohomology representation of $\omega\in H^{2n+1}_\dR(X,Y)$ in \eqref{22sept2021}, we can also  assume that 
% $\alpha^{j}=0,\ \ j=0,1,2,\cdots,  n-1$. The argument is as follows. The restriction of $\U$ to $Y$ is also a good covering, but might have more than $n+1$ open sets. We know that relative to  $\U|_Y$ we have $H^{2n}(Y,\Omega^0|_Y)=0$. This cohomology contains an element represented by $\alpha^0$.  
% This implies that we can modify $\omega$ with a $D_{X,Y}$ exact element such that $\alpha^0$ becomes zero.
% This process can be continued until getting $\alpha^{n-1}=0$. It follows from the closedness of $\omega$ that $(-1)^{n+1}d\omega^n+\delta(\alpha^n)=0$. 
% Therefore, if we define $\check H^{2n+1}_\dR(X)$ similar as above,  its dimension might be less than the dimension of $H^{2n+1}_\dR(X)$ and we may not have a decomposition of $H^{2n+1}_\dR(X,Y)$.  However, if we consider its $F^{n+1}$ piece, then $\omega^n=0$ and $\delta\alpha^n=0$. We have now $(0,\alpha)$ which is in $\check H^{2n}_\dR(Y)$ and we get the desired decomposition.
% \end{proof}
Note that using the long exact sequences of the pair $(X,Y)$ both in homology and cohomology we have 
\begin{equation}
\label{02.12.2021}
\int_{\delta}\tilde \alpha=\int_{Z_2-Z_1}\alpha, \ \ \ \ \ \forall\alpha\in H^{2n}_\dR(Y),    
\end{equation}
and $\tilde\alpha$ is the image of $\alpha$ under $H^{2n}_\dR(Y)\to H^{2n+1}_\dR(X,Y)$.  Let $Z_s, s\in (\C,0)$ be a family of algebraic cycles as in Introduction. 
We fix points $s_0,s \in (\C,0)$, and a path $\gamma$ which connects $s_0$ to $s$ in $(\C,0)$.
Let $n$ be the complex dimension of the fibers and $Y_s=Z_s\cup Z_{s_0}$. 
We consider $\omega\in \check F^{n+1} H^{2n+1}_\dR(X)$ and 
$$
I(s):=\int_{Z_{s_0}}^{Z_{s}}\omega:=\int_{f^{-1}(\gamma)}\omega.
$$
Since $\omega$ is a closed differential $2n+1$, this integral does only depends on the homotopy class of $\gamma$. By the definition of Hodge filtration, the Cech cohomology representation of $\omega\in F^{n+1}H^{2n+1}(X, Y_s)$ in \eqref{22sept2021} starts with $(\omega^{n+1},\alpha^n)$. Moreover, its  pull-back by  $g$ contains only the term $(\omega^{n+1},\alpha^n)$. By definition  of $\check F^{n+1}H^{2n+1}_\dR(X)$ we have $\alpha^n=0$ and define the Gelfand-Leray form in our context as 
$$
\frac{\omega}{df}:=\frac{\omega^{n+1}}{df}\in H^{2n}_\dR(Z_s). 
$$
\begin{prop}
\label{22s2021}
We have 
$$
I'(s)=\int_{Z_s}\frac{\omega}{df}. 
$$
\end{prop}
\begin{proof}
We use Fubini's theorem:
$$
I(s)=\int_{f^{-1}(\gamma)}\frac{\omega}{df}\wedge ds=\int_{\gamma} \left(\int_{Z_s}\frac{\omega}{df} \right)ds.
$$
\end{proof}
\begin{rem}\rm
If instead of $(\C,0)$ we use  an algebraic (compact) curve $C$, $I(s)$ is a multi valued 
function with branching points in the set of atypical values of $f$. We denote it by  $A\subset C$ (it also includes the singularities of $C$). The difference of two branches of $I$ is the period $\int_{f^{-1}(\gamma)}\omega$, where $\gamma$ is a closed path in $C\backslash A$.
In this case Proposition \ref{22s2021} implies that  $I'(s)$ is a meromorphic function on $C$. 
\end{rem}

 \section{Proof of Theorem \ref{main2021}}
 \label{3.10.2021}
 We will use Griffiths description of the cohomology of hypersurfaces in \cite{gr69}.
 Let $\T$ be the parameter space of smooth hypersurfaces of degree $d$ in $\P^{2n+2}$, $F=\sum_{t\in\T} t_{\alpha}x^\alpha$ be the universal homogeneous polynomial of degree $d$ in $x=(x_0,x_1,\cdots,x_{2n+2})$ and $\X:=\{F=0\}\subset \P^{2n+2}\times \T$.  We define
 $$
 \omega_P=\Resi \frac{P\Omega}{F^k},\ \ \deg(P)=k\cdot d-(2n+3),\ \ \Omega:=\sum_{i=0}^{2n+2}(-1)^{i}x_i\widehat{dx_i}
 $$
 For the case $(2n+3)|d$ we also define $\omega:=\omega_1$ ($1$ refers to the constant polynomial $1$). This gives us $\omega_P\in H^{2n+1}(X)$, and since $H^{2n+1}_\dR(X,Y)\to H^{2n+1}_\dR(X)$ is surjective it comes from an element in $H^{2n+1}_\dR(X,Y)$ which we denote it again by $\omega_P$. By Proposition \ref{10.09.2021France}, for $\deg P=kd-(2n+3),\  k\leq n+1$ this element is unique as 
 $F^{n+1}H^{2n+1}_\dR(X,Y)\to F^{n+1}H^{2n+1}_\dR(X)$ is an isomorphism.
 %and by Proposition \ref{portadosfundos2021} for $k=(n+1)$ we can choose a unique $\omega_P\in H^{2n+1}(X,Y)$. 
 
 Consider the case $n=1$ and $d=5$. Assume that Clemens' conjecture is not true. It follows that for a component $\tilde\W$ of $\W$, the map $\pi:\tilde \W\to\T$ in Section \ref{Hilbert2021} is surjective, and hence, $\V=\T$. Moreover, the irreducible components of fibers of $\pi$ have dimension $\geq 3$ (recall that they are ${\rm PSL }(2,\C)$-invariant).  We take $t_0\in\T$ which is still a generic point of $\T$, and we cannot assume that  $X_{t_0}$ is the Fermat variety. 
 We get a holomorphic family of algebraic curves $Z_{1,t}, Z_{2,t}\subset X_t,\ \ t\in (\T,0)$ such that $Z_{1,t}$ and $Z_{2,t}$ are  strongly algebraically equivalent. We use Proposition \ref{10.09.2021Ethan} and we get 
 \begin{equation}
 \label{7.9.2021}
 \int_{\delta}\omega=0. 
 \end{equation}
 Note that by Proposition \ref{10.09.2021France}, $F^3H^3(X,Y)$ is one dimensional and it is generated by $\omega$. 
 We write $\omega$ in a covering and  it is of the form $(\omega^3,0)$, where $\omega_3=\{\omega_{3,i}\}$.
 Note that the element $\omega^3$ restricted to $\Y$ is identically zero, and hence in Theorem \ref{7set2021bolsonaro} the corresponding $\alpha$ can be taken $0$. Since $\omega$ is $D_{X,Y}$-closed, 
 we have  $d\omega=dt\wedge *$ and $\delta\omega=dt\wedge *$.   
 Consider a derivation $\frac{\partial}{\partial t}$ in $\T$, for instance take $F=F_0+tx^i$, where $x^i$ is a homogeneous monomial of degree $d$ in $x$. Since $\tilde\W\to\T$ is surjective, we can have such derivations for all monomials $x^i$ of degree $d$.  We have 
 \begin{equation}
  d\left(  \frac{P\Omega}{F^k}  \right)=dt\wedge   \frac{-k\frac{\partial F}{\partial t}P\Omega}{F^{k+1}}
 \end{equation}
 where $d$ is the differential in the variety $(\P ^4\times \T)\backslash \X$. This can be seen easily after passing to an affine coordinate. Taking residue we have
 $$
 d\omega_P=dt\wedge \omega_{-k \frac{\partial F}{\partial t}P}
 $$
We make the derivation of \eqref{7.9.2021} with respect to $t$, use  Theorem  \ref{7set2021bolsonaro} and conclude that
 \begin{equation}
 \label{10012022}
 \int_{\delta}\omega_{P}=0,\ \ \forall P\in\C[x]_5. 
 \end{equation}
 %Note that $\Y$ is a two dimensional surface, and $\omega|_\Y$ is zero, and hence, its integration over $\Zc_i$ is zero. 
 In the Cech cohomology representation of $\omega_P$ as above, it is of the form $(\omega^3,0)+(\omega^2,0)$ because 
 $$
 D_{X,Y}(\omega^3,0)=(d\omega^3,0)+(\delta\omega^3,0).
 $$ 
 Therefore, $\omega_P$ is in $F^2 H^3_\dR(X)$. %defined in Proposition \ref{portadosfundos2021}.  
 We use Proposition \ref{22s2021} and conclude that  $\int_{Z_s}\frac{\omega_P}{ds}=0$.  
 \begin{rem}\rm
 A more constructive way to prove Clemens' conjecture, is to give an example of a smooth quintic threefold with isolated  rational curves.
 In this paper we have generalized a proof of Noether's theorem, and T. Shioda in \cite{Shioda1981a} has given an explicit example of a surface of prime degree and 
 with Picard rank equal to one, and hence, he has given a constructive proof of Noether's theorem. The following quintic is a natural generalization of Shioda's example
 $$
x_0^d+x_1x_2^{d-1}+x_2x_3^{d-1}+x_3x_4^{d-1}+x_4x_1^{d-1}=0,\ \ d=5, 
$$
and since $5$ is a prime number, we may expect that Clemens' conjecture holds for this quintic
 \end{rem}
 %In this section we pose a problem and give a partial solution to it. Here, it will become clear why rational curves must be  isolated in a generic quintic threefold, whereas curves of degree five are not. 
%\section{De Rham cohomology of Fermat varieties and Carlson-Grifffiths theorem}
\section{Computing period of curves}
\label{04.12.2021-2}
In order to gather evidences that the periods \eqref{3oct2021} do not vanish identically in $s$, one may try to compute them for well-known families of rational curves inside smooth quintic threefolds, see for instance \cite{Straten2012, Walcher2012, Mustata2021}. In this section, we describe how to compute the period \eqref{3oct2021} and perform this for one family of rational curves. 

Let $X\subseteq \P^{m+1}$ be a smooth degree $d$ hypersurface given by $X=\{F=0\}$ and
$$
\Omega=\sum_{i=0}^{m+1}(-1)^ix_idx_0\wedge\cdots\widehat{dx_i}\cdots\wedge dx_{m+1}.
$$
Let $P$ be a polynomial of degree $d(q+1)-m-2)$ with $q\in \{0,1,\ldots,m\}$. 
By the computations in 
\cite{CarlsonGriffiths1980} we know that 
\begin{equation}
\label{5.9.2018.2}
{\rm Resi} \left(\frac{P\Omega}{F^{q+1}}\right)=\frac{(-1)^{m}}{q!}\left\{\frac{P\Omega_J}{F_J}\right\}_{|J|=q}\in H^q(\U,\Omega_{X}^{m-q}),
\end{equation}
where $\Omega_J:=\iota_{\frac{\partial}{\partial x_{j_q}}}(\cdots\iota_{\frac{\partial}{\partial x_{j_0}}}(\Omega)\cdots)$,  $F_J:=F_{j_0}\cdots F_{j_q}$, $F_j:=\frac{\partial F}{\partial x_i}$, $j_0<j_1<\cdots<j_q$ and $\U=\{U_i\}_{i=0}^{m+1}$ is the Jacobian covering restricted to $X$ given by $U_i=\{F_i\neq 0\}\cap X$.
We have also 
$$
\Omega_J=(-1)^{j_0+j_1+\cdots+j_q+\binom{q+2}{2}}\sum_{l=0}^{m-q}(-1)^lx_{k_l}\widehat{dx_{k_l}},
$$
where $k_0<k_1<\cdots<k_{m-q}$ is obtained form $0,1,2,\ldots,m+1$ by removing $j_0,j_1,\ldots,j_q$, see for instance \cite{Roberto2022}. 
We apply this theorem for the Fermat variety of dimension $m=3$ and degree $d=5$ and with $q=1$. In this case the Jacobian covering is the usual covering by $U_i:=\{x_i\not=0\}$. For a homogeneous polynomial $P$ of degree $5$ up to constant independent of $X$,  we have 
$$
\omega_P:={\rm Resi} \left(\frac{P\Omega}{F^{2}}\right)=\left\{\sum_{j_2\not= j_0,j_1 }\frac{(-1)^{j_2^*}Px_{j_2}}{x_{j_0}^4x_{j_1}^4}dx_{j_3}\wedge dx_{j_4}\right\}
\in H^1(\U,\Omega_{X}^{2}),
$$
where the sum runs over $j_2\in \{0,1,\ldots,4\}\backslash\{j_0,j_1\}$ with $\{j_0,j_1,\ldots,j_4\}=\{0,1,\ldots,4\}$, and $j_2^*$ is the position of $j_2$ in $0,1,2,3,4$ with $j_0,j_1$ removed (counting from zero).  Let $[x_0(s,t):x_1(s,t):\cdots: x_4(s,t)],\ \ s\in(\C,0),\ t\in\P^1$ be a family of rational curves 
inside the Fermat quintic threefold.  We write $x_i(s,t)=x_i(t)+sy_i(t)+O(s)$ and 
$dx_{j_3}\wedge dx_{j_4}=(x_{j_3}'y_{j_4}-x_{j_4}'y_{j_3})dt\wedge ds+O(s)$. We conclude that
$$
\int_{Z_s}\frac{\omega_P}{ds}=\sum_{j_0,j_1}
{\rm Resi}_{x_{j_0}=0}\left(\frac{P}{x_{j_0}^4x_{j_1}^4}
\sum_{j_2\not= j_0,j_1} (-1)^{j_2^*}x_{j_2}(x_{j_3}'y_{j_4}-x_{j_4}'y_{j_3} )dt
\right)
$$
where ${\rm Resi}_{x_{j_0=0}}$ refers to sum of the residues around  the points of $x_{j_0}=0$ in $\P^1$. Note that for $x_i=x_i(\frac{at+b}{ct+d})$ with $a,b,c,d$ depending  on $s$, 
the expression $ x_{j_2}'y_{j_3}-x_{j_3}'y_{j_2} $ is identically zero because $
\frac{\partial}{\partial s}x_i(s,t)=x_i'(t)\frac{\partial}{\partial s}\left( \frac{at+b}{ct+d}\right)$. This is natural because the action of ${\rm PSL}(2,\C)$ on $\P^1$ does not produce families of rational curves. Note also that $y_0x_0^4+y_1x_1^4+\cdots+y_4x_4^4$.  
%\begin{eqnarray*}
% & & x_0^5+x_1^5+\cdots+x_4^5=0\\
% & & y_0x_0^4+y_1x_1^4+\cdots+y_4x_4^4=0
%\end{eqnarray*}
Families of lines in the Fermat quintic fourfold has been studied in \cite{AlbanoKatz}. It has $50$ one dimensional families of the form 
$$
[x:y]\mapsto [x: -\zeta y: ay:by:cy], \ \zeta^5=1,\ a^5+b^5+c^5=0. 
$$
We set $x=t, y=1$ and write this as $x(t,s)=[t: -\zeta: 1:s:(-s^5-1)^{\frac{1}{5}}],\ s\in (\C,0)$. We have $y=[0:0:0:1:0], x'=[1:0:0:0:0]$ and so the only non-trivial possibility for $(j_3,j_4)$ is $(0,3)$. Therefore, we  have $(j_0,j_1,j_2)=(1,2,4),(1,4,2),(2,4,1)$ and hence up to some $2\pi i$ factor 
$$
\int_{Z_s}\frac{\omega_P}{ds}=
{\rm Resi}_{x_1=0}\left(\frac{P}{x_1^4x_2^4}x_4dt
\right)+
{\rm Resi}_{x_1=0}\left(\frac{P}{x_1^4x_4^4}(-x_2)dt
\right)+
{\rm Resi}_{x_2=0}\left(\frac{P}{x_{2}^4x_{4}^4}(-x_1)dt
\right).
$$
The third residue is zero as $x_2=-\zeta$ is constant. We set $P=x_1^3x_2^2$ and as $x_2$ and $x_4$ are independent of $t$ and we have
$$
\int_{Z_s}\frac{\omega_P}{ds}=\frac{x_4}{x_2^2}-\frac{x_2^3}{x_4^4}=
\frac{(-s^5-1)^{\frac{1}{5}} }{\zeta^2}+\frac{\zeta^3}{(-s^5-1)^{\frac{4}{5}} }.
$$
It would be also nice to make the computations for the following  family of conics described in \cite{Mustata2021}:
$$
a^2(x_0 +x_1)=b^2(x_2 +x_3),\ 
bx_4 =c(x_0 +x_1),\ b(x_0^2+x_1^2)=\pm ia(x_2^2 +x_3^2)=0, a^{10}+b^{10}-4b^{5}c^5=0. 
$$

\section{Abel-Jacobi map}
\label{10jan2022}
\def\Jac{{\rm Jac}}
Let $Z_s,\ s\in \S$ be a family of curves in a smooth threefold $X$. We fix a curve $Z_{s_0}$ and the Abel-Jacobi map is defined as 
$$
\S\to \Jac(X):=\frac{(F^2H^3_\dR(X))^\vee}{H_3(X,\Z)},\ s\mapsto  \int_{Z_{s_0}}^{Z_s}F^2H^3_\dR(X),
$$
where the integration takes over the homology $\delta$ between $Z_{s_0}$ and $Z$: $\partial \delta=Z_s-Z_{s_0}$. For a proof of the existence of such a homology path see \cite[Proposition 19.1.1, page 373]{Fulton1998}, however, note that from this proof it is not clear whether $\delta$ is supported in $\cup_{s\in \S}Z_s$ or not.   If the parameter space $\S=\P^1$ is the rational curve, it is well-known that there is no non constant holomorphic map from $\P^1$ to a complex compact tori, and hence, the Abel-Jacobi map is zero in this case. This implies that for some homology $\delta$ between $Z_{s_0}$ and $Z_s$ we have 
in $\int_{\delta}F^2=0$, but not necessarily $\delta$ is supported in $\cup_{s\in \P^1}Z_s$.  
There are examples of this situation for which we have this stronger statement.
\begin{prop}
Let $Z\subset X$ be an irreducible divisor  in $X$ with a singular set of codimension $\geq 2$ and such that a desingularization $\tilde Z$ of $Z$ satisfies $H_1(\tilde Z,\Q)=0$. Then for any rational function $f$ in $Z$ and family of curves $Z_s:=f^{-1}(s),\ s\in\P^1$ we have $\int_{Z_{s_0}}^{Z_s}F^2=0$, where the integration takes place over $f^{-1}(\gamma)$ and $\gamma$ is any path in $\P^1$ connecting $s_0$ to $s$.  
\end{prop}
\begin{proof}
 The pull-back  $\pi^*\omega$ of $\omega\in H^3_\dR(X)$ by the desingularization map $\pi: \tilde Z\to Z\subset X$ is an element in $H^3_\dR(\tilde Z)$ which in turn by hard Lefschetz theorem is isomorphic to $H^1_\dR(\tilde Z)$. By our hypothesis this is zero and so $\pi^*\omega=D\eta$, for some $2$-cocycle $\eta\in \oplus_{i=0}^2 C^i(\U,\Omega^{2-i}_{\tilde Z})$. Here, we are using the notation of hypercohomology relative to a good covering $\U$ of $\tilde Z$.  Let $\tilde f=f\circ\pi$ and $\tilde Z_s:=\tilde f^{-1}(s)$. Since a desingularization is a biholomorphism over a Zariski open subset and no component of a fiber of $f$ is inside the singular set of $Z$, it is enough to prove that $\int_{\tilde Z_{s_0}}^{\tilde Z_s}F^2=0$. By Stokes theorem we have 
 $$
 \int_{\tilde Z_{s_0}}^{\tilde Z_s}\pi^*\omega=\int_{\tilde Z_s}\eta-
 \int_{\tilde Z_{s_0}}\eta
 $$
 The map $\P^1\to \C, s\mapsto \int_{Z_s}\eta$ is holomorphic, and hence, constant.  
\end{proof}
\begin{rem}\rm
Note that if $\dim(X)>2$ then by the long exact sequence of $0\to\Z\to \O_X\to\O_X^*\to 0 $ and $H^1(X,\O_X)=H^2(X,\O_X)=0$, we know that $H^1(X, \O_X^*)\cong\Z$ and any divisor is a hypersurface section (not necessarily transversal section) of $X$. By Lefschetz theorem a transversal (and hence smooth) hyperplane section $Z$ of $X$ satisfies $H_1(Z,\Q)=0$. \end{rem}

\begin{rem}\rm 
In the proof of Theorem \ref{main2021} we have the vanishing integral   \eqref{10012022} for $\delta$ which is supported in the family of cycles $Z_s,\ s\in\C$. This implies that in this context the Abel-Jacobi map is zero. Theorem \ref{main2021}  is stronger than this. 
\end{rem}

\def\cprime{$'$} \def\cprime{$'$} \def\cprime{$'$} \def\cprime{$'$}

%\bibliography{biblio}
%%\bibliography{../Biblio/biblio}
%\bibliographystyle{alpha}
\end{document}